\documentclass[12pt]{article}
\usepackage[T1]{fontenc}
\usepackage[utf8]{inputenc}
\usepackage{lmodern}
\usepackage{fullpage}
\usepackage{amsmath,amssymb,amsthm}
\usepackage{enumitem,booktabs,microtype}
\usepackage{tikz,tikz-cd}
\usetikzlibrary{arrows.meta,positioning,fit,backgrounds,calc}
\usepackage[colorlinks=true,linkcolor=blue!45!black,citecolor=blue!45!black,
  urlcolor=blue!45!black]{hyperref}
\hypersetup{pdftitle={Linear Laurent phenomenon algebras are of finite type},
  pdfauthor={Oliver Daisey}}

\newtheorem{thm}{Theorem}[section]
\newtheorem{lem}[thm]{Lemma}
\newtheorem{prop}[thm]{Proposition}
\newtheorem{cor}[thm]{Corollary}
\theoremstyle{definition}
\newtheorem{defn}[thm]{Definition}
\newtheorem{example}[thm]{Example}
\newtheorem{remark}[thm]{Remark}
\numberwithin{equation}{section}
\DeclareMathOperator{\Frac}{Frac}
\DeclareMathOperator{\diag}{diag}
\DeclareMathOperator{\UB}{UB}
\DeclareMathOperator{\pol}{pol}
\newcommand{\A}{\mathcal A}
\newcommand{\F}{\mathcal F}
\newcommand{\Lr}{\mathcal L}
\newcommand{\U}{\mathcal U}
\newcommand{\Z}{\mathbb Z}

\newcommand{\hf}[1]{\widehat{f^{#1}}}
\newcommand{\doi}[1]{\href{https://doi.org/#1}{doi:\nolinkurl{#1}}}
\setlist[enumerate]{itemsep=0.2em,topsep=0.3em}
\setlist[itemize]{itemsep=0.2em,topsep=0.3em}
\tikzset{vertex/.style={circle,draw,fill=white,minimum size=7mm,inner sep=1pt},
  edge/.style={-{Stealth[length=2mm]},semithick},
  seed/.style={draw,rounded corners=2pt,fill=blue!3,inner sep=6pt}}

\begin{document}
\begin{center}
  {\Large Linear Laurent phenomenon algebras\\[1mm]
  are of finite type}\\[5mm]
  {\small Oliver Daisey}\\[1mm]
  {\small\href{mailto:oliverjdaisey@gmail.com}{\texttt{oliverjdaisey@gmail.com}}}
\end{center}

\begin{abstract}
Laurent phenomenon algebras generalise cluster algebras by allowing
irreducible exchange polynomials in place of binomials. Lam and Pylyavskyy
proved finite type for those arising from graphs and conjectured that the
same holds for every Laurent phenomenon algebra admitting a linear seed.
We prove this conjecture over an arbitrary unique factorisation domain.
Our proof constructs a nonzero element of the upper algebra divisible by every
cluster variable. Comparing noninitial cluster variables with the prime
factors of this element in the initial Laurent ring bounds the number of
cluster variable classes and yields a finite normalised mutation pattern.
\end{abstract}

\tableofcontents
\section{Introduction}
Laurent phenomenon algebras (LP algebras) were introduced by Lam and
Pylyavskyy in 2012 \cite{lam2012laurent} to generalise Fomin and Zelevinsky's
cluster algebras. Introduced in \cite{fomin2002cluster}, cluster algebras
provide a combinatorial framework for studying total positivity and canonical
bases. Their generators are grouped into clusters, and mutation exchanges one
variable at a time using an exchange binomial. Among the central results of
the theory are the Laurent phenomenon and the finite-type classification by
Dynkin diagrams \cite{fomin2002cluster2}.

LP algebras allow arbitrary irreducible exchange polynomials in place of
binomials. This greater freedom retains the Laurent phenomenon, but makes the
mutation rule and the question of finite type more delicate. In rank two,
Lam and Pylyavskyy give a complete finite-type classification
\cite[Theorem~6.4]{lam2012laurent}. In higher rank, they prove finite type for
the graph LP algebras, whose initial exchange polynomials have independent constant
terms and variable coefficients in $\{0,1\}$ \cite{lam2016linear}. They
conjecture more generally that every LP algebra admitting a linear seed is of
finite type \cite[Section~1.7 and Conjecture~7.5]{lam2016linear}.

In this paper, we prove this conjecture. We set the notation
$S=(\mathbf x,\mathbf f)$ for a seed and write a linear initial exchange polynomial as
\begin{equation}\label{eq:linear}
 f^\ell=a_0^\ell+\sum_{t\ne\ell}a_t^\ell x_t,
 \qquad a_t^\ell\in A,
\end{equation}
where $A$ is a unique factorisation domain. Here, and throughout, ``linear''
means total degree at most one in the cluster variables. This includes
irreducible constant exchange polynomials. The coefficients may vanish or satisfy
relations, provided that the resulting polynomials satisfy the seed axioms.

\begin{remark}
One might try to prove the result first for generic coefficients and
then introduce relations. However, specialisation can introduce common factors
in the cancellation step of mutation and change the denominators of the
exchange Laurent polynomials. A proof for the generic case therefore does
not automatically give a proof for all linear seeds. 
\end{remark}

To state our result more precisely, put
\begin{equation}\label{eq:matrix-intro}
 z_\ell=\frac{f^\ell}{x_\ell},\qquad
 C_{\ell t}=\begin{cases}a_t^\ell&\ell\ne t,\\0&\ell=t,\end{cases}
 \qquad M=\diag(z_1,\ldots,z_n)-C,
 \qquad Y_I=\det M_{I,I}.
\end{equation}
We set $Y_\varnothing=1$. The symbols $z_\ell$ denote ordinary ratios of
exchange polynomials to their corresponding cluster variables; they need not be the variables obtained by mutating $x_\ell$, which use the exchange Laurent polynomials.
The minors $Y_I$ are weighted generalisations of Lam and Pylyavskyy's
graph LP polynomials \cite[Section~4.2, Proposition~4.5]{lam2016linear}.
Write $\U=\bigcap_{T}\Lr(T)$ for the upper LP algebra, where the intersection
is over the seeds in the mutation pattern of $S$.

\begin{thm}\label{thm:main}
Every linear LP algebra over a unique factorisation domain is of finite type.
More precisely, for a linear seed \eqref{eq:linear}, the following hold.
\begin{enumerate}
\item Every cluster variable not associated over $A$ to an initial variable
divides a nonzero $Y_I$ in $\U$.
\item Let $Q$ denote the product of all nonzero principal minors of $M$, that is,
\begin{equation}\label{eq:QH}
 Q=\prod_{\substack{\varnothing\ne I\subseteq[n]\\Y_I\ne0}}Y_I.
\end{equation}
Set $H = (x_1\cdots x_n)Q$. Then $H/v\in\U$ for every cluster variable $v$.
\item If $s$ is the number of distinct irreducible factors of $Q$ in
$A[x_1^{\pm1},\ldots,x_n^{\pm1}]$, up to units of this Laurent ring, then
there are at most $n+s$ cluster-variable classes up to $A^\times$ and at most
$\binom{n+s}{n}$ seed-equivalence classes. These admit a finite normalisation.
\end{enumerate}
\end{thm}

We first recall the basic definitions given in \cite{lam2012laurent}, taking care to distinguish ordinary
exchange polynomials, exchange Laurent polynomials, and choices of coefficient units.
Section~\ref{sec:upper} shows that cluster variables are prime as elements of the upper
algebra, and constructs polynomial residues along them.
Section~\ref{sec:determinants} uses the directed coefficient graph to detect
vanishing minors of the matrix $M$. In Section~\ref{sec:finite}, an argument based on the maximum degree of the polynomial residues
is used to prove Theorem~\ref{thm:main}.

\section{Preliminaries}\label{sec:preliminaries}
We first recall the basic definitions. A more detailed treatment is available
in \cite{lam2012laurent}. Let $A$ be a commutative unique factorisation domain,
let $K=\Frac(A)$, and let $\F=K(x_1,\ldots,x_n)$, where $n\geq1$.
Examples of coefficient rings include $\Z$, polynomial rings over $\Z$ or a
field, and their localisations. We impose no restriction on the
characteristic. We write $g\sim_A h$ and say that $g$ and $h$ are \emph{associated} if $g=uh$ for some $u\in A^\times$.

\subsection{Seeds and exchange Laurent polynomials}
\begin{defn}[Seed]
A \emph{seed} is a tuple $S=(\mathbf x,\mathbf f)$, where
$\mathbf x=\{x_1,\ldots,x_n\}$ is a set of algebraically independent
generators of $\F$ over $K$, and
$\mathbf f=\{f^1,\ldots,f^n\}\subset A[x_1,\ldots,x_n]$ satisfies:
\begin{enumerate}[label=(LP\arabic*)]
\item Each $f^\ell$ is irreducible and is not divisible by a cluster variable.
\item The polynomial $f^\ell$ does not involve $x_\ell$.
\end{enumerate}
Here irreducible means nonzero, a nonunit, and not a product of two nonunits.
We call $\mathbf x$ the \emph{cluster}, its elements \emph{cluster variables},
and the $f^\ell$ \emph{exchange polynomials}. The rank of the seed is $n$.
\end{defn}

The pairs $(x_\ell,f^\ell)$ are considered unordered, but their correspondence is part
of the data. We use labels when recording mutation sequences. 

\begin{defn}[Exchange Laurent polynomials]\label{def:hat}
The exchange Laurent polynomial associated to $f^\ell$ has the form
\begin{equation}\label{eq:hat}
 \hf\ell=f^\ell\prod_{t\ne\ell}x_t^{-e_{\ell t}},
 \qquad e_{\ell t}\in\Z_{\geq0}.
\end{equation}
The exponents are uniquely determined by requiring, for every $t\ne\ell$,
that
\[
 \hf\ell\big|_{x_t\leftarrow f^t/\xi}
 \in A[x_1^{\pm1},\ldots,x_{t-1}^{\pm1},\xi^{\pm1},
             x_{t+1}^{\pm1},\ldots,x_n^{\pm1}]
\]
and that this element is not divisible by $f^t$ in that ring.
\end{defn}

Equivalently, $e_{\ell t}$ is the largest power of $f^t$ dividing
$f^\ell|_{x_t\leftarrow f^t/\xi}$ in that Laurent ring. The existence and uniqueness of these Laurent polynomials are established in
\cite[Lemmas~2.3--2.4]{lam2012laurent}.

\begin{lem}\label{lem:denominator}
If $e_{\ell t}>0$, then $f^t$ does not involve $x_\ell$.
\end{lem}
\begin{proof}
Write $f^\ell=\sum_{r\geq0}b_rx_t^r$, with $b_r$ independent of $x_t$.
By (LP1), $b_0\ne0$. Divisibility of
$\sum_rb_r(f^t)^r\xi^{-r}$ by $f^t$ implies divisibility of its constant
coefficient $b_0$. This divisibility can be checked in the polynomial ring, since
localising at coordinate variables does not invert any factor of $f^t$.
But $b_0$ is independent of $x_\ell$ by (LP2). Additivity of polynomial
degree over a domain shows that its factor $f^t$ is also independent of
$x_\ell$.
\end{proof}

For a nonzero Laurent polynomial $G$ in a cluster $\mathbf q$, we shall use
the notation
\begin{equation}\label{eq:pol}
 \pol_{\mathbf q}(G)=G\prod_jq_j^{-m_j},\qquad
 m_j=\min\{\text{exponent of }q_j\text{ in a monomial of }G\}.
\end{equation}
This is the unique coordinate-monomial multiple of $G$ which is polynomial
and has no coordinate factor; in other words, it is the device that normalises exchange Laurent polynomials to exchange polynomials. In particular, it does not rescale coefficients.

\subsection{Mutation}
The important difference from the cluster setting is that an LP exchange
relation uses $\hf i$, and the subsequent exchange polynomials are defined
only up to coefficient units.

\begin{defn}[Mutation]\label{def:mutation}
	Mutation at index $i$ replaces $S$ by $S_i=(\mathbf y_i,\mathbf f_i)$. The procedure operates as follows. The cluster is
\[
 \mathbf y_i=\{x_1,\ldots,x_{i-1},y_i,x_{i+1},\ldots,x_n\},
 \qquad x_iy_i=\hf i.
\]
We refer to the variables $x_i$, $y_i$ as \emph{mutation partners}. We set $f_i^i=f^i$. For $\ell\ne i$, if $f^\ell$ is independent of $x_i$,
we take $f_i^\ell\sim_A f^\ell$. Otherwise, we perform the following three steps:
\begin{enumerate}
\item \emph{Substitution.} Put
\[
 Z=\hf i\big|_{x_\ell\leftarrow0},\qquad
 G=f^\ell\big|_{x_i\leftarrow Z/y_i}.
\]
\item \emph{Cancellation.} In the Laurent ring in the new coordinates,
take $G$ and remove all occurrences of each prime factor of $Z$.
\item \emph{Normalisation.} Take the polynomial representative
\eqref{eq:pol} of the result. This gives $f_i^\ell$ up to a unit of $A$.
\end{enumerate}
\end{defn}

Lemma~\ref{lem:denominator} ensures that $\hf i$ has no negative power of
$x_\ell$ in the substitution step. Moreover, $Z\ne0$ because $f^i$ has no
factor $x_\ell$. The cancellation removes entire prime powers occurring in
$G$, not merely the powers occurring in a single greatest common divisor
with $Z$.

We write
\[
 S_{s_1\cdots s_k}=\mu_{s_k}\circ\cdots\circ\mu_{s_1}(S)
\]
and denote the exchange polynomial at position $\ell$ by $f^\ell_{s_1\cdots s_k}$.
When needed, $y_{s_1\cdots s_k}$ denotes the variable inserted by the final
mutation. Equality after a mutation sequence
will be stated either for specified representatives or up to seed
equivalence. 

\begin{prop}[Basic mutation properties]\label{prop:mutation}
Mutation produces a valid seed and is reversible with suitable choices of
units. Moreover, $f^\ell$ depends on $x_i$ if and only if $f_i^\ell$ depends
on $y_i$, and $\widehat{f_i^i}=\hf i$.
\end{prop}
\begin{proof}
These are the mutation properties established in
\cite[Section~2.2, in particular Lemma~2.11 and
Propositions~2.15--2.16]{lam2012laurent}.
\end{proof}

\subsection{LP algebras, equivalence, and finite type}
\begin{defn}[Seed equivalence]
	Two seeds are \emph{equivalent} if, after appropriately matching cluster variables to exchange polynomials in each seed, corresponding variables and exchange polynomials are
associated over $A$ as elements of $\F$. We write $[S]$ for the equivalence
class of $S$.
\end{defn}
When coordinates have been rescaled, their polynomial expressions are
interpreted in the common ambient field. Mutation respects this equivalence
\cite[Lemma~3.1]{lam2012laurent}.

\begin{defn}[Laurent phenomenon algebra]
A Laurent phenomenon algebra $\A=(\A,\mathbf S)$ consists of a connected
mutation pattern of seeds $\mathbf S$ in $\F$, closed under mutation with
specified unit choices, together with the $A$-subalgebra $\A\subseteq\F$
generated by its cluster variables. The notation $\A(S)$ indicates a chosen
initial seed $S$ in this pattern.
\end{defn}
Equivalently one may first work with the mutation class up to seed
equivalence. Multiplying generators by coefficient units does not change
the generated $A$-algebra. Connectedness excludes a disjoint union of
unrelated mutation classes.

\begin{prop}[Laurent phenomenon]\label{prop:laurent}
For any seed $T$ of an LP algebra $\A$, every cluster variable is an element
of $\Lr(T)$. In particular, $\A\subseteq\Lr(T)$.
\end{prop}
\begin{proof}
This is \cite[Theorem~5.1]{lam2012laurent}.
\end{proof}

\begin{defn}[Normalisation and finite type]\label{def:finite}
A pattern is \emph{normalised} if it contains exactly one representative of
each of its seed-equivalence classes. A \emph{normalisation} of a pattern is
a mutation-compatible choice of equivalent representative seeds. An LP
algebra is \emph{of finite type} if it admits a normalisation with finitely
many seeds.
\end{defn}
Note that choosing arbitrary representatives does not by itself check closure under
the exact mutation rule $f_i^i=f^i$. We shall make a compatible choice in
Proposition~\ref{prop:normalisation}. 

In rank two, if the exchange degrees are $b,c\geq0$, finite
type is equivalent to $bc<4$ \cite[Theorem~6.4]{lam2012laurent}.

\section{The upper algebra and polynomial residues}\label{sec:upper}
In this section we establish some notions for arbitrary LP algebras. The purpose is to let us
replace questions about an entire mutation pattern by calculations in a
seed and its immediate neighbours.

\subsection{A finite intersection of Laurent rings}
\begin{defn}
The \emph{upper LP algebra} and the \emph{upper bound at a seed $T$} are
\[
 \U=\bigcap_{T'\in\mathbf S}\Lr(T'),\qquad
 \UB(T)=\Lr(T)\cap\bigcap_{i=1}^n\Lr(\mu_iT).
\]
All intersections take place in $\F$. The Laurent phenomenon gives
$\A\subseteq\U\subseteq\UB(T)$.
\end{defn}
The Laurent ring is unchanged by coefficient-unit rescalings, hence these
definitions are independent of the representatives used for neighbouring
seeds.

We recall precisely the local input from the LP caterpillar lemmas. Consider
an alternating path in two distinct directions, suppressing the unchanged
coordinates:
\begin{equation}\label{eq:caterpillar}
 (x,y)\longrightarrow(z,y)\longrightarrow(z,u)\longrightarrow(v,u).
\end{equation}
If the four Laurent rings are $\Lr_0,\Lr_1,\Lr_2,\Lr_3$, then
\begin{equation}\label{eq:caterpillar-input}
 u,v\in\Lr_0,\qquad
 \gcd_{\Lr_0}(z,u)=\gcd_{\Lr_0}(z,v)=1.
\end{equation}
These are \cite[Lemmas~5.2 and 5.4]{lam2012laurent}. These lemmas also apply when exchange polynomials coincide or differ only by a coefficient unit. 
We note that $z$ is prime in $\Lr_0$, since it differs from the exchange polynomial at $x$ by a Laurent unit.

\begin{lem}[Caterpillar intersection]\label{lem:intersection}
For the path \eqref{eq:caterpillar},
\[
 \Lr_1\cap\Lr_3\subseteq\Lr_0,
 \qquad \Lr_0\cap\Lr_2\subseteq\Lr_3.
\]
\end{lem}
\begin{proof}
If $h\in\Lr_1\cap\Lr_3$, then, since $x,y$ and the suppressed coordinates
are units in $\Lr_0$, we may write
\[
 h=\frac{a}{z^r}=\frac{b}{u^sv^t},\qquad
 a,b\in\Lr_0,\quad r,s,t\geq0.
\]
The equality $au^sv^t=bz^r$, primality of $z$, and
\eqref{eq:caterpillar-input} imply $z^r\mid a$ in $\Lr_0$. Thus
$h\in\Lr_0$. Reversing the path proves the other inclusion.
\end{proof}

\begin{prop}[Upper-bound invariance]\label{prop:upper}
For every seed $T$, one has $\UB(T)=\U$.
\end{prop}
\begin{proof}
Let $T'=\mu_kT$. For $j\ne k$, apply Lemma~\ref{lem:intersection} to
\[
 \mu_jT\xrightarrow{j}T\xrightarrow{k}T'\xrightarrow{j}\mu_jT'.
\]
It gives
\[
 \Lr(\mu_jT)\cap\Lr(T')\subseteq\Lr(\mu_jT'),\qquad
 \Lr(T)\cap\Lr(\mu_jT')\subseteq\Lr(\mu_jT).
\]
Together with the two common rings $\Lr(T)$ and $\Lr(T')$, these inclusions
show that $\UB(T)=\UB(T')$. In rank one the equality is immediate from
involutivity. Connectedness now shows that this common upper bound lies in
every Laurent ring of the pattern, and hence in $\U$. The reverse inclusion
is part of the definition.
\end{proof}

This is Fei's LP starfish theorem \cite[Theorem~1.3]{fei2026}. We have included
the short argument to make clear that no assumption $\hf\ell=f^\ell$,
and no pairwise coprimality assumption on the exchange polynomials is needed.
The weaker invariance statement under $\hf\ell=f^\ell$ also appears in
\cite{du2022some}.

\subsection{Cluster variables are prime in the upper algebra}
Fix a seed with coordinates $\mathbf q=\{q_1,\ldots,q_n\}$ and exchange polynomials
$g^1,\ldots,g^n$. Let $q_i'$ be the mutation partner of $q_i$, and write
\[
 \Lr_i=A[q_1^{\pm1},\ldots,q_{i-1}^{\pm1},(q_i')^{\pm1},
                   q_{i+1}^{\pm1},\ldots,q_n^{\pm1}].
\]
The exchange polynomial $g^i$ is a prime nonunit in $\Lr_i$. Denote its prime
valuation on $\F$ by $\nu_i$. This records the exponent of $g^i$ in a
factorisation of a nonzero rational function in the Laurent UFD $\Lr_i$.
Since $q_iq_i'=\widehat{g^i}$, we have
\begin{equation}\label{eq:valuation-coordinates}
 \nu_i(q_j)=\delta_{ij}.
\end{equation}

\begin{lem}\label{lem:prime}
Every cluster variable is prime in $\U$, and $\U^\times=A^\times$.
More precisely,
\begin{equation}\label{eq:prime-ideal}
 q_i\U=\U\cap g^i\Lr_i,
 \qquad \nu_i(h)>0\ \Longleftrightarrow\ h\in q_i\U
 \quad(0\ne h\in\U).
\end{equation}
\end{lem}
\begin{proof}
The coordinate $q_i$ is a unit in the Laurent ring of its seed and in every
first neighbouring ring except $\Lr_i$. In $\Lr_i$, it is a unit times $g^i$.
Consequently, if $h\in\U\cap g^i\Lr_i$, then $h/q_i$ belongs to every
Laurent ring defining the upper bound at this seed. By
Proposition~\ref{prop:upper}, $h/q_i\in\U$. The other containment in
\eqref{eq:prime-ideal} is immediate. The contracted ideal on the right is
proper and prime, so $q_i$ is a prime nonunit of $\U$. The valuation
criterion follows as well.

If $h\in\U^\times$, its expression as a unit of $\Lr(T)$ is
$h=\epsilon\prod_jq_j^{\beta_j}$, with $\epsilon\in A^\times$ and
$\beta_j\in\Z$. Both $h$ and $h^{-1}$ belong to $\Lr_i$. Equation
\eqref{eq:valuation-coordinates} gives $\beta_i\geq0$ and $-\beta_i\geq0$
for every $i$. Thus $h=\epsilon$. Conversely every coefficient unit is
plainly a unit of $\U$.
\end{proof}

\subsection{Polynomial residues}
Let $v$ be any cluster variable, let $w_1,\ldots,w_{n-1}$ be the other
coordinates of a seed containing it, and let $v'$ be its mutation partner.
Put
\begin{equation}\label{eq:residue-chart}
 B=A[w_1^{\pm1},\ldots,w_{n-1}^{\pm1}],\qquad
 vv'=E=mP,
\end{equation}
where $P\in B$ is the exchange polynomial and $m\in B^\times$ is the
normalising monomial. The polynomial $P$ is a prime nonunit of $B$, so
$B/(P)$ is an integral domain.

\begin{lem}[Polynomial residues]\label{lem:residues}
Reduction modulo $P$ in the neighbouring chart $B[(v')^{\pm1}]$ induces a
homomorphism
\begin{equation}\label{eq:rho}
 \rho:\U\longrightarrow (B/(P))[T],\qquad
 \ker\rho=v\U,\qquad \rho(v')=T,
\end{equation}
where $T$ is an indeterminate.
\end{lem}
\begin{proof}
For $h\in\U$, write $h=\sum_{r\in\Z}b_rv^r$ in $B[v^{\pm1}]$, with
only finitely many nonzero $b_r\in B$. In the neighbouring chart,
\[
 h=\sum_r b_rE^r(v')^{-r}.
\]
Since $h\in B[(v')^{\pm1}]$, uniqueness of Laurent coefficients implies
$b_rE^r\in B$ for every $r$. Reduce coefficients modulo $P$, and denote
the image of $v'$ by $T$. Each term $b_rE^r(v')^{-r}$ then maps to
$\overline{b_rE^r}T^{-r}$. For $r>0$, the coefficient is divisible by $P$
and hence vanishes, so only terms with $r\leq0$ survive. The image is
therefore polynomial, not merely Laurent, in $T$.

The kernel is $\U\cap PB[(v')^{\pm1}]=v\U$ by
Lemma~\ref{lem:prime}.
\end{proof}

We note in particular that $\U/v\U$ embeds in a polynomial ring in one variable.

\begin{example}\label{ex:residue}
In the rank-two seed over $\Z$ with exchange polynomials $1+x_2$ and $1+x_1$, set
$v=x_1$, $v'=(1+x_2)/x_1$, and $B=\Z[x_2^{\pm1}]$.
Reduction modulo $P=1+x_2$ gives
\[
 \rho(x_1)=0,\qquad \rho(x_2)=-1,\qquad \rho(v')=T.
\]
The cluster variable $(1+x_1+x_2)/(x_1x_2)=(1+v')/x_2$ has residue
$-(1+T)$. 
\end{example}

\section{Linear seeds and their principal minors}\label{sec:determinants}
We now return to the fixed linear seed \eqref{eq:linear}. Recall $z_\ell=f^\ell/x_\ell$ and $M=\diag(z_1,\ldots,z_n)-C$ from
\eqref{eq:matrix-intro}. Thus the defining equations for the matrix are
\begin{equation}\label{eq:linear-system}
 x_\ell z_\ell=a_0^\ell+\sum_{t\ne\ell}a_t^\ell x_t,
 \qquad
 M\begin{pmatrix}x_1\\\vdots\\x_n\end{pmatrix}
 =\begin{pmatrix}a_0^1\\\vdots\\a_0^n\end{pmatrix}.
\end{equation}
The initial cluster variables on the right of the definition of $z_\ell$
are fixed rational coordinates, so $Y_I$ is a well-defined element of $\F$.

\begin{remark}
This construction specialises to the determinants of Lam and
Pylyavskyy \cite[Section~4.2, Proposition~4.5]{lam2016linear}. In their graph
setting, $a_0^\ell=A_\ell$ are independent coefficient indeterminates and
$a_t^\ell$ is $1$ or $0$ according as $\ell\to t$ is an edge or not. Their
matrix has diagonal entries $(A_\ell+\sum_{\ell\to t}x_t)/x_\ell$ and
off-diagonal entries $-1$ on edges. Its principal minors are their
acyclic-function generating functions $Y_I$, with the same notation used
here. For strongly connected $I$, these are precisely the noninitial
cluster variables \cite[Theorem~1.1]{lam2016linear}. The graph-case $Y_I$
also occur in the cluster-monomial basis theorem of
Dantas e Moura, Telekicherla Kandalam, and Woodruff \cite{dantas2025}.
\end{remark}

For arbitrary coefficients, the determinants can split or vanish. Our
additional assertion is that every noninitial cluster variable divides a
nonzero $Y_I$ in $\U$ (Theorem~\ref{thm:minor-divisibility}); we do not assert
that these weighted minors are themselves always cluster variables.

\begin{lem}\label{lem:minors-upper}
Every $z_\ell$, and hence every $Y_I$, lies in $\A\subseteq\U$.
\end{lem}
\begin{proof}
If $y_\ell$ is obtained by mutating $x_\ell$, then
\[
 z_\ell=\frac{f^\ell}{x_\ell}
 =y_\ell\prod_{t\ne\ell}x_t^{e_{\ell t}}\in\A.
\]
The determinant $Y_I$ is a polynomial over $A$ in the $z_\ell$, so it also
belongs to $\A$. Apply the Laurent phenomenon.
\end{proof}

\begin{example}\label{ex:ranktwo-minors}
For $f^1=1+x_2$ and $f^2=1+x_1$ over $\Z$,
\[
 M=\begin{pmatrix}(1+x_2)/x_1&-1\\-1&(1+x_1)/x_2\end{pmatrix}.
\]
Its three nonempty principal minors are
\[
 Y_{\{1\}}=\frac{1+x_2}{x_1},\qquad
 Y_{\{2\}}=\frac{1+x_1}{x_2},\qquad
 Y_{\{1,2\}}=\frac{1+x_1+x_2}{x_1x_2}.
\]
In this example the minors themselves are precisely the noninitial cluster
variables. In general it is their prime factors that matter.
\end{example}

\subsection{The directed coefficient graph}
Define a directed graph $\Gamma$ on $[n]\cup\{0\}$ by
\[
 \ell\longrightarrow t\ \Longleftrightarrow\ a_t^\ell\ne0
 \quad(t\ne\ell),\qquad
 \ell\longrightarrow0\ \Longleftrightarrow\ a_0^\ell\ne0.
\]
The vertex $0$ has no outgoing edges. For graph calculations only, put
$x_0=1$. A subset is strongly connected if any two of its vertices are
joined in each direction by directed paths contained in the subset.

\begin{remark}[Rooted-forest interpretation]\label{rem:forests}
Put $N=M\diag(x_1,\ldots,x_n)$. Give $\ell\to t$ weight
$w_{\ell t}=a_t^\ell x_t$ for $t\ne0$, and $\ell\to0$ weight
$w_{\ell0}=a_0^\ell$. The augmented matrix
\[
 \widetilde N=
 \begin{pmatrix}N&-(a_0^\ell)_{\ell\in[n]}\\0&0\end{pmatrix}
\]
is the row Laplacian of this weighted graph. Applying the principal-minor
case of the all-minors matrix-tree theorem \cite{chaiken1982}, with the
transpose convention for a row Laplacian, gives
\begin{equation}\label{eq:forest-expansion}
 \left(\prod_{\ell\in I}x_\ell\right)Y_I
 =\det N_{I,I}
 =\sum_{F\in\mathcal T_I}\ \prod_{(\ell\to t)\in F}w_{\ell t}.
\end{equation}
Here $\mathcal T_I$ consists of spanning directed forests rooted at
$([n]\cup\{0\})\setminus I$.
Thus the determinantal
numerators are weighted rooted-forest polynomials. The identity holds over
$A$ by specialisation of the polynomial identity in formal edge weights.
The next proposition additionally rules out total cancellation when the
coefficients satisfy relations; we give a direct leading-term argument.
\end{remark}

\begin{defn}\label{def:closed}
A \emph{closed homogeneous component} is a nonempty strongly connected
subset $J\subseteq[n]$ such that
\begin{equation}\label{eq:closed}
 a_0^\ell=0\quad(\ell\in J),\qquad
 a_t^\ell=0\quad(\ell\in J,\ t\notin J).
\end{equation}
Equivalently, it is a sink strongly connected component of $\Gamma$ other
than $\{0\}$.
\end{defn}
Here ``homogeneous'' means that the exchange polynomials in these rows have no constant
terms, and ``closed'' means that they involve no coordinates outside $J$.

\begin{prop}[Nonvanishing criterion]\label{prop:nonzero}
For a nonempty $I\subseteq[n]$, the following conditions are equivalent:
\begin{enumerate}
\item $Y_I\ne0$ as a rational function in the initial coordinates;
\item every vertex of $I$ has a directed path to $([n]\cup\{0\})\setminus I$;
\item $I$ contains no closed homogeneous component.
\end{enumerate}
\end{prop}
\begin{proof}
If a vertex cannot reach the complement of $I$, all vertices reachable from
it lie in $I$. A sink strongly connected component of this reachable
subgraph has no outgoing edge in the full graph and is a closed homogeneous
component. Conversely, a vertex of a closed component cannot leave it.
This proves the equivalence of (2) and (3).

Suppose $J\subseteq I$ is a closed homogeneous component. Order the vertices
of $J$ before those of $I\setminus J$. The block in the rows of $J$ and
columns of $I\setminus J$ is zero, so $M_{I,I}$ is block triangular.
Furthermore, \eqref{eq:linear-system} gives
\[
 M_{J,J}(x_\ell)_{\ell\in J}=0.
\]
This is a nonzero kernel vector over $\F$. Thus $Y_J=0$ and $Y_I=0$.

It remains to show that (2) implies (1), including when the coefficients
have signs or satisfy relations. Consider the polynomial matrix
\begin{equation}\label{eq:Nmatrix}
 N_I=M_{I,I}\diag(x_t:t\in I),\qquad
 (N_I)_{\ell\ell}=f^\ell,\quad
 (N_I)_{\ell t}=-a_t^\ell x_t\ (\ell\ne t).
\end{equation}
For $\ell\in I$, let $d_\ell$ be the shortest directed distance to the
complement of $I$, and put $d_t=0$ outside $I$, including at $0$.
Choose pairwise distinct real numbers $\epsilon_1,\ldots,\epsilon_n$ in
$(0,1/3)$ and assign weights
\[
 \omega_t=-d_t+\epsilon_t\quad(t\ne0),\qquad \omega_0=0
\]
to $x_t$. These are auxiliary weights on formal monomials, not values of
the coefficients. Every neighbour of $\ell$ has distance at least
$d_\ell-1$, and one has distance exactly $d_\ell-1$. Hence there is a unique
neighbour $t(\ell)$ of greatest weight and its distance is $d_\ell-1$.

Let $e_\ell$ denote the standard row vectors with positions indexed by $I$.
Row $\ell$ of $N_I$ is the sum, over its outgoing edges, of
\[
 a_t^\ell x_t(e_\ell-e_t)\quad(t\in I),\qquad
 a_t^\ell x_te_\ell\quad(t\notin I).
\]
The same formula includes the edge to $0$ since $x_0=1$.
Choose in each row the summand indexed by $t(\ell)$. Its edge strictly
decreases distance. In an ordering by distance, the matrix of the selected
row vectors is triangular with diagonal entries one. Thus its determinant
is one. By multilinearity in the rows, the unique choice of largest total
weight contributes
\begin{equation}\label{eq:leading-tree}
 \prod_{\ell\in I}a_{t(\ell)}^\ell x_{t(\ell)}
\end{equation}
to $\det N_I$. Every other choice has smaller total weight. The coefficient
in \eqref{eq:leading-tree} is nonzero because $A$ is an integral domain. This term
cannot cancel, so $\det N_I\ne0$ and $Y_I\ne0$.
\end{proof}

We note this proof requires neither positive
coefficients nor characteristic zero. 

\subsection{Cofactor ratios in the upper algebra}
A vanishing determinant cannot detect a prime divisor. We next show that
its proper cofactors provide the additional regular functions needed in
the proof of finite type.

\begin{lem}\label{lem:cofactor}
Let $J$ be a closed homogeneous component and let $i\in J$. Then
\begin{equation}\label{eq:cofactor}
 Y_{J\setminus\{i\}}\ne0,\qquad
 Z_{J,i}:=\frac{Y_{J\setminus\{i\}}}{x_i}\in\U.
\end{equation}
\end{lem}
\begin{proof}
A singleton closed component would have $f^i=0$, contrary to (LP1). Hence
$|J|\geq2$, and $I=J\setminus\{i\}$ is nonempty. Every vertex of $I$ can
reach $i$ inside $J$, so Proposition~\ref{prop:nonzero} gives $Y_I\ne0$.
Since the exchanges in the rows of $J$ involve only coordinates in $J$,
\[
 D_I:=\Bigl(\prod_{t\in I}x_t\Bigr)Y_I
       =\det N_I\in A[x_t:t\in J].
\]
The sum of row $\ell\in I$ of $N_I$ is $a_i^\ell x_i$.
After setting $x_i=0$, every row sum is zero. The nonempty matrix then has
the all-ones vector in its kernel, so its determinant is zero.
It follows that $x_i\mid D_I$ in the polynomial ring. Thus
\begin{equation}\label{eq:cofactor-laurent}
 Z_{J,i}=\frac{D_I/x_i}{\prod_{t\in I}x_t}
\end{equation}
is polynomial in $x_i$ and Laurent in the other initial coordinates.

In the first neighbouring chart at $i$, substitute $x_i=\hf i/y_i$ into
\eqref{eq:cofactor-laurent}. Since there is no negative power of $x_i$,
the result is Laurent. In every other first neighbouring chart, $Y_I$ is
Laurent by Lemma~\ref{lem:minors-upper} and $x_i$ is an unchanged unit.
Hence $Z_{J,i}\in\UB(S)=\U$ by Proposition~\ref{prop:upper}.
\end{proof}

Notice that the conclusion is membership in the upper algebra. We need
neither $Z_{J,i}$ to be a cluster variable nor an equality $\A=\U$.

\section{Divisibility and finite type}\label{sec:finite}
We now combine the polynomial residues with the linear relations of the
initial seed. The essential observation is that coordinates of maximum
residue degree force a principal submatrix to be singular.

\subsection{Every noninitial variable divides a nonzero minor}
\begin{thm}\label{thm:minor-divisibility}
Let $S$ be the linear seed \eqref{eq:linear}. If a cluster variable $v$ is
not associated over $A$ to any initial coordinate, then
\[
 Y_I\ne0,\qquad \frac{Y_I}{v}\in\U
\]
for some nonempty $I\subseteq[n]$.
\end{thm}
\begin{proof}
Suppose, for a contradiction, that $v$ divides no nonzero principal minor
in $\U$. Choose a seed containing $v$, let $v'$ be its mutation partner,
and use the notation $B,P$ of \eqref{eq:residue-chart}. Compose the map of
Lemma~\ref{lem:residues} with the inclusion into $k[T]$, where
$k=\Frac(B/(P))$. Denote the composite again by $\rho$, so that
\[
 \rho:\U\longrightarrow k[T],\qquad \ker\rho=v\U,
 \qquad \rho(v')=T.
\]
Write
\[
 X_\ell=\rho(x_\ell),\qquad W_\ell=\rho(z_\ell).
\]
All these polynomials are nonzero. Indeed, $X_\ell=0$ would imply
$v\mid x_\ell$ in $\U$. By Lemma~\ref{lem:prime}, both are prime, so they
would be associated by a unit of $\U$, hence by a unit of $A$, contrary to
the choice of $v$. Also $z_\ell=Y_{\{\ell\}}\ne0$, so $W_\ell=0$ would
contradict the assumption that $v$ divides no nonzero minor.

At least one $X_\ell$ has positive degree in $T$. To see this, observe that
\begin{equation}\label{eq:localise}
 \U[x_1^{-1},\ldots,x_n^{-1}]
 =A[x_1^{\pm1},\ldots,x_n^{\pm1}]=\Lr(S).
\end{equation}
The two inclusions follow from $A[\mathbf x]\subseteq\U\subseteq\Lr(S)$.
Since every $X_\ell\ne0$, $\rho$ extends to this localisation with values
in $k(T)$. If all the $X_\ell$ were constant, the image of $\Lr(S)$ would
lie in $k$. But the Laurent phenomenon puts $v'$ in $\Lr(S)$, whereas
$\rho(v')=T\notin k$. This is impossible.

Set
\begin{equation}\label{eq:maximum-degree}
 d=\max_\ell\deg_T X_\ell>0,\qquad
 I=\{\ell:\deg_T X_\ell=d\}.
\end{equation}
Let bars denote the images of coefficients in $k$. Applying $\rho$ to
\eqref{eq:linear-system} gives
\begin{equation}\label{eq:residue-linear-system}
 X_\ell W_\ell=\overline{a_0^\ell}
       +\sum_{t\ne\ell}\overline{a_t^\ell}X_t.
\end{equation}
For $\ell\in I$, the right side has degree at most $d$. The left side is
a nonzero product, and $\deg_T X_\ell=d$. Thus
\begin{equation}\label{eq:constant-W}
 W_\ell\in k^\times\qquad(\ell\in I).
\end{equation}
If $b_\ell\in k^\times$ is the leading coefficient of $X_\ell$ for
$\ell\in I$, comparing coefficients of $T^d$ in
\eqref{eq:residue-linear-system} yields
\[
 W_\ell b_\ell=
 \sum_{t\in I\setminus\{\ell\}}\overline{a_t^\ell}b_t.
\]
Therefore $\rho(M_{I,I})(b_\ell)_{\ell\in I}=0$, and
\begin{equation}\label{eq:singular-residue}
 \rho(Y_I)=0.
\end{equation}
If $Y_I\ne0$ in the original ambient field, then
\eqref{eq:singular-residue} says $v\mid Y_I$ in $\U$, the desired
contradiction.

It remains to deal with $Y_I=0$ identically. This is different from merely
having zero residue. By Proposition~\ref{prop:nonzero}, there is a closed
homogeneous component $J\subseteq I$. For any $i\in J$,
Lemma~\ref{lem:cofactor} supplies
\begin{equation}\label{eq:cofactor-identity}
 x_iZ_{J,i}=Y_{J\setminus\{i\}},\qquad
 Z_{J,i}\in\U,\qquad Y_{J\setminus\{i\}}\ne0.
\end{equation}
By our assumption, the right side has nonzero residue. Its residue is also
constant: its diagonal entries are among $W_\ell$ for $\ell\in I$, and
\eqref{eq:constant-W} applies. Reducing \eqref{eq:cofactor-identity} gives
\[
 X_i\rho(Z_{J,i})\in k^\times.
\]
But $\deg_T X_i=d>0$ and $\rho(Z_{J,i})\in k[T]$. A polynomial of positive
degree cannot multiply another polynomial to give a nonzero constant.
This final contradiction proves the theorem.
\end{proof}

\begin{cor}\label{cor:common-multiple}
The element $H$ in \eqref{eq:QH} is nonzero and belongs to $\A$.
For every cluster variable $v$, $H/v\in\U$.
\end{cor}
\begin{proof}
All retained factors belong to $\A$ and are nonzero. If $v\sim_A x_i$, use
the factor $x_i$ of $H$. Otherwise Theorem~\ref{thm:minor-divisibility}
supplies a factor $Y_I$ of $Q$ with $Y_I/v\in\U$ and we multiply by the remaining
factors of $H$.
\end{proof}

\subsection{Counting cluster variables}
We stress that no unique factorisation property for $\U$ is being assumed.
Factorisation will take place in the initial Laurent unique factorisation domain
$\Lr_0=A[x_1^{\pm1},\ldots,x_n^{\pm1}]$.
As before, we let $\nu_i$ be the first-neighbour valuations for the initial seed (see
\eqref{eq:valuation-coordinates}).

\begin{lem}\label{lem:localised-primes}
A noninitial cluster variable $v$ is a prime nonunit in $\Lr_0$ and satisfies
$\nu_i(v)=0$ for every $i$. If two noninitial cluster variables are associated
in $\Lr_0$, then they are associated over $A$.
\end{lem}
\begin{proof}
By Lemma~\ref{lem:prime}, $v$ and the $x_i$ are prime in $\U$, and $v$ is
not associated to any $x_i$. Hence $v$ divides no product of initial
coordinates. The prime ideal $v\U$ survives the localisation
\eqref{eq:localise}, so $v$ remains a prime nonunit in $\Lr_0$.

Likewise $x_i$ does not divide $v$ in $\U$. Since $v\in\U$, its valuation
$\nu_i(v)$ is nonnegative, and \eqref{eq:prime-ideal} shows that it is zero.
If $v=\epsilon x_1^{\beta_1}\cdots x_n^{\beta_n}w$ for another noninitial
variable $w$, with $\epsilon\in A^\times$, applying $\nu_i$ gives
$0=\beta_i$. Thus $v=\epsilon w$.
\end{proof}

\begin{prop}\label{prop:variable-count}
If $Q$ has $s$ distinct irreducible factors up to Laurent units, there are
at most $n+s$ cluster-variable classes up to $A^\times$.
\end{prop}
\begin{proof}
Every noninitial cluster variable divides $Q$ in $\U$, and hence in
$\Lr_0$, by Theorem~\ref{thm:minor-divisibility}. It is prime there by
Lemma~\ref{lem:localised-primes}, so it is associated to one of the $s$
irreducible factors of $Q$. The last assertion of the lemma says that two
different coefficient-unit classes cannot correspond to the same Laurent
prime. The initial variables contribute the remaining $n$ classes.
\end{proof}

\subsection{Seeds are determined by their clusters}
For completeness, we include the form of Du and Li's seed-rigidity theorem
\cite[Theorem~3.1]{du2022some} that we need.

\begin{lem}[Seed rigidity]\label{lem:rigidity}
Two seeds in one LP algebra whose clusters agree up to coefficient units
are equivalent.
\end{lem}
\begin{proof}
First suppose the clusters agree exactly, with coordinates
$\mathbf q=\{q_1,\ldots,q_n\}$. Write the exchanges of the two seeds as
$g^i,h^i$, and their mutation partners at $i$ as
\[
 u=\frac{\widehat{g^i}}{q_i},\qquad
 w=\frac{\widehat{h^i}}{q_i}.
\]
Set $B_i=A[q_j^{\pm1}:j\ne i]$. Laurentness of $u$ in the neighbouring
chart $B_i[w^{\pm1}]$ and the equality
\[
 u=w\,\frac{\widehat{g^i}}{\widehat{h^i}}
\]
show that $\widehat{g^i}/\widehat{h^i}\in B_i$: this is the single
coefficient of $w$ in the Laurent expansion. Interchanging the two seeds
shows that the reciprocal is also in $B_i$. The ratio is therefore a unit
of $B_i$, namely a coefficient unit times a Laurent monomial in the $q_j$.
Taking polynomial representatives gives $g^i\sim_A h^i$ for every $i$.

If the clusters agree only up to coefficient units, express both sets of
exchange polynomials in the same rescaled coordinates. The Laurent rings
are unchanged, and the two displayed exchange relations acquire only
coefficient-unit factors. The same argument applies and gives equivalence
in the common ambient field.
\end{proof}

\subsection{A compatible choice of representatives}
\begin{prop}\label{prop:normalisation}
If an LP algebra of rank $n$ has $N<\infty$ cluster-variable classes, then it admits a normalisation with at most $\binom{N}{n}$ seeds.
\end{prop}
\begin{proof}
Choose one representative of each variable class. Every genuine cluster
determines a set $\mathcal C$ of $n$ chosen representatives. Coordinates
in one cluster cannot be associated, by algebraic independence.
Lemma~\ref{lem:rigidity} says that $\mathcal C$ determines a single
seed-equivalence class. Mutation compatibility therefore determines, for
each $u\in\mathcal C$, the class of its mutation partner. Let $v$ be the
chosen representative of that class.

The product $uv$ is a coefficient-unit multiple of the exchange Laurent
polynomial in that direction, expressed in the coordinates
$\mathcal C\setminus\{u\}$. Define
\begin{equation}\label{eq:chosen-exchange}
 g_{\mathcal C,u}=\pol_{\mathcal C}(uv).
\end{equation}
These exchanges give a valid seed $S_{\mathcal C}$ equivalent to the
original class. Indeed, in each direction we have just rescaled coordinates and
multiplied an ordinary exchange polynomial by a coefficient unit.
The normalisation exponents are unaffected by these unit changes.
Consequently, the exchange Laurent polynomial of
$g_{\mathcal C,u}$ is exactly $uv$, and mutating $u$ inserts exactly $v$.

Put $\mathcal C'=(\mathcal C\setminus\{u\})\cup\{v\}$. By mutation
compatibility, a mutation of $S_{\mathcal C}$ at $u$ is equivalent to
$S_{\mathcal C'}$. In all other directions the unit choices in mutation
allow us to select precisely the exchanges of $S_{\mathcal C'}$.
At the replaced coordinate the polynomial is required to remain unchanged.
This too holds exactly, since $uv$ is independent of both $u$ and $v$, and the
other coordinates are unchanged. Hence
\[
 g_{\mathcal C,u}=\pol_{\mathcal C}(uv)
 =\pol_{\mathcal C'}(vu)=g_{\mathcal C',v}.
\]
Thus these representative seeds form a mutation-closed pattern. It is
connected as the image of the original pattern and has at most
$\binom{N}{n}$ seeds.
\end{proof}

\begin{proof}[Proof of Theorem~\ref{thm:main}]
Theorem~\ref{thm:minor-divisibility} proves (1), and
Corollary~\ref{cor:common-multiple} proves (2).
Proposition~\ref{prop:variable-count} gives at most $n+s$ variable classes.
Proposition~\ref{prop:normalisation} supplies a finite normalisation and
the seed bound in (3). This proves finite type.
\end{proof}

\section*{Acknowledgements}
The author would like to thank Anna Felikson for useful discussions
surrounding this problem, and Tom Ducat for insights that helped shape
the approach taken in this paper.

\bibliographystyle{alpha}
\bibliography{references}
\end{document}